\documentclass[11pt,reqno]{amsart}
\usepackage{amssymb,amsthm,amsfonts,amsmath,amscd,amssymb,epsfig}
\usepackage[all]{xy}
\usepackage[normalem]{ulem}

\usepackage{enumerate,array,hyperref, graphicx,color,amsthm,
bbm,mathrsfs, a4wide}

\theoremstyle{plain}
\newtheorem{theorem}{Theorem}[section]

\newtheorem{corollary}[theorem]{Corollary}

\newtheorem{proposition}[theorem]{Proposition}

\newtheorem{lemma}[theorem]{Lemma}
\newtheorem*{question}{Question}

\theoremstyle{remark}

\newtheorem*{remark}{Remark}

\theoremstyle{definition}
\newtheorem*{definition}{Definition}
\newcommand{\id}{{{\mathchoice {\rm 1\mskip-4mu l} {\rm 1\mskip-4mu l}
{\rm 1\mskip-4.5mu l} {\rm 1\mskip-5mu l}}}}

\newcommand{\om}{\omega}

\newcommand{\w}{\wedge}

\newcommand{\Z}{{\mathbb{Z}}}
\newcommand{\R}{{\mathbb{R}}}
\newcommand{\C}{{\mathbb{C}}}

\renewcommand{\max}{{\rm max}}

\newcommand{\sgn}{{\rm sgn\,}}

\newcommand{\comment}[1]{}

\newcommand{\beq}{\begin{equation}}
\newcommand{\beqn}{\begin{equation}\nonumber}
\newcommand{\eeq}{\end{equation}}

\newcommand{\bea}{\begin{equation}\begin{aligned}}
\newcommand{\bean}{\begin{equation}\begin{aligned}\nonumber}
\newcommand{\eea}{\end{aligned}\end{equation}}

\newcommand{\RP}{\mathbb{RP}}
\newcommand{\Sb}{\overline{\Sigma}}

\newcommand{\uh}{\hat{u}}
\newcommand{\vh}{\hat{v}}

\begin{document}

\title{The Conley-Zehnder indices of the rotating Kepler problem}
\author{Peter Albers}
\author{Joel W.~Fish}
\author{Urs Frauenfelder}
%\author{Helmut Hofer}
\author{Otto van Koert}
\address{
	Peter Albers\\
	Mathematisches Institut\\
	Westf\"alische Wilhelms-Universi\"at M\"unster
	}
\email{peter.albers@uni-muenster.de}
\address{
	Joel W.~Fish\\
	Max Planck Institute\\
    Leipzig
	}
\email{joel.fish@mis.mpg.de}
%\address{
%	Helmut Hofer\\
%	Institute for Advanced Study
%	}
%\email{hofer@ias.edu}
\address{
    Urs Frauenfelder\\
    Department of Mathematics and Research Institute of Mathematics\\
    Seoul National University}
\email{frauenf@snu.ac.kr}
\address{
    Otto van Koert\\
    Department of Mathematics and Research Institute of Mathematics\\
    Seoul National University}
\email{okoert@snu.ac.kr}


\maketitle

\begin{abstract}
We determine the Conley-Zehnder indices of all periodic orbits of the rotating Kepler problem for energies below the critical Jacobi energy. Consequently, we show the universal cover of the bounded component of the regularized energy hypersurface is dynamically convex. Moreover, in the universal cover there is always precisely one periodic orbit with Conley-Zehnder index 3, namely the lift of the doubly covered retrograde circular orbit.
\end{abstract}

%\tableofcontents

\section{Introduction}

The Kepler problem in rotating coordinates arises as the limit of the planar circular restricted 3-body problem when the mass of one of the primaries goes to zero, and hence serves as an approximation of the restricted planar 3-body problem for a small mass parameter.
The ultimate goal is to study the dynamics of the 3-body problem using finite energy foliations.
One essential ingredient is the so-called Conley-Zehnder index of a periodic orbit.
These indices play a central role in the theory of finite energy foliations, symplectic field theory, Fukaya $A_\infty$-categories, and various Floer theories.

In this article we completely determine the Conley-Zehnder indices of all periodic orbits for all energies below the (unique) critical value of the Jacobi energy in the regularized system. The upshot is that on every energy hypersurface there exists precisely one periodic orbit of Conley-Zehnder index 1, namely the simply-covered retrograde circular orbit. This orbit is non-contractible. Moreover, there exists a unique contractible periodic orbit of Conley-Zehnder index 3, namely the doubly-covered retrograde circular orbit. In particular, for energies below the critical value of the Jacobi energy, the universal cover of an energy hypersurface in the rotating Kepler problem is dynamically convex, i.e.~contractible periodic Reeb orbits have Conley-Zehnder index at least $3$.

We point out that due to the $S^1$ action on the rotating Kepler problem most periodic orbits are in fact degenerate. The Conley-Zehnder index we consider in this article is the one from \cite{HWZ_the_dynamics_on_three_dimensional_strictly_conve_eneergy_surfaces} which is lower semi-continuous and, in the non-degenerate situation, equals the transversal Conley-Zehnder index.

\begin{theorem}\label{thm:main}
For energies below the critical value of the Jacobi energy, the bounded component of an energy hypersurface of the rotating Kepler system is dynamically convex. Moreover, there is precisely one periodic orbit with Conley-Zehnder index $3$, namely the doubly covered retrograde circular orbit.
\end{theorem}

The standard way to construct the universal cover of an energy hypersurface in the regularized rotating Kepler problem (and in fact the regularized restricted planar three body problem) as $S^3\subset\C^2$ is via the Levi-Civita embedding, see \cite{Levi_Civita}. In \cite[Theorem 3.4]{HWZ_the_dynamics_on_three_dimensional_strictly_conve_eneergy_surfaces} Hofer-Wysocki-Zehnder prove that a strictly convex $S^3\subset\C^2$ is automatically dynamically convex, which in turn guarantees that the Conley-Zehnder index of each orbit is at least three. However, this result does not directly apply to the Levi-Civita embedding:

\begin{theorem}\label{thm:main2}
The image of the Levi-Civita embedding of the regularized rotating Kepler problem is not convex for energies close to the critical value of the Jacobi energy.
\end{theorem}

\begin{question}
Is the universal cover of the bounded component of the regularized restricted planar three body problem for energies below the first critical value dynamically convex?
\end{question}

\begin{remark}
Above the first critical value the regularized restricted planar three body problem is not dynamically convex due to the existence of the Lyapunov orbits which have Conley-Zehnder index 2.

In \cite{Albers_Fish_Frauenfelder_Hofer_Koert_Global_surfaces_of_sectio_for_the_3BP} we proved that for large mass ratios and for sufficiently negative energy levels the answer to the above question is "yes." Moreover we proved that for these mass ratios and energy levels the Levi-Civita embedding is actually convex.
Theorem \ref{thm:main} asserts that for mass ratio 0 the answer is ``yes'' again. 
Furthermore we checked numerically whether the Levi-Civita embedding is convex by discretizing the energy hypersurface and testing the tangential Hessian for negative eigenvalues.
These numerical results suggest that the Levi-Civita embedding is convex for most mass ratios. Furthermore, in case convexity fails, the measure of the set of non-convex sample points is very small. Thus, we tend to believe that the answer to the above question concerning dynamical convexity is ``yes'' for all mass ratios.
\end{remark}

{\bf Acknowledgments:} We thank E.~Belbruno, B.~Bramham and H.~Hofer for stimulating discussions.
The research of  P.~Albers, J.~Fish was partially supported by the NSF-grants
DMS-0903856, DMS-0802927. U.~Frauenfelder was partially supported by the Basic Research fund 2010-0007669 and O.~van Koert by the New Faculty Research Grant 0409-20100147 funded by the Korean government.

P.~Albers, J.~Fish, U.~Frauenfelder and O.~van Koert thank the IAS for its hospitality.

\section{Kepler laws}

The Kepler problem is given by the Hamiltonian $E(q,p):T^*(\R^2\setminus\{0\})\cong\big(\R^2\setminus\{0\}\big)\times\R^2\to\R$
\beq
E(q,p):=\tfrac12 |p|^2-\frac{1}{|q|}\;.
\eeq

Since $E$ is invariant under rotations in $\R^2$ around $0$ the angular momentum
\beq
L:=q_1p_2-q_2p_1
\eeq
is a first integral of the motion. For negative energies $E<0$ the solutions of the Hamiltonian equations are either ellipses or collision orbits. The eccentricity $\epsilon$ of an ellipse is given by
\beq\label{eqn:eccentricity}
\epsilon^2=2EL^2+1\;.
\eeq
For $\epsilon=0$ the ellipse is a circle and for $\epsilon\to1$ the ellipse degenerates into a collision orbit. According to Kepler's 3rd law we have the equality
\beq\label{eqn:Kepler_period_energy}
T^2=-\frac{\pi^2}{2E^3}
\eeq
where $T$ denotes the period of the ellipse.

\section{Moser regularization of the (inertial) Kepler problem}\label{sec:Moser_regularization_of_the_inertial_Kepler_problem}

The flow of the Kepler problem is periodic outside the set of collision orbits. However, it is well known that double collisions can be regularized. A nice description of the regularization is given by Moser in \cite{Moser_Regularization_of_Keplers_problem_and_the_averaging_method_on_a_manifold}, which embeds the Kepler flow for negative energies into the geodesic flow on the 2-sphere.

Let us illustrate this with the following example.
Consider the energy level $E=-\tfrac12$ and set
\beq
K(q,p):=|q|\big(E(q,p)+\tfrac12)+1=\tfrac12\big(|p|^2+1)|q|\;.
\eeq
Since $|q|\neq0$ the flows associated to $E$ and $K$ coincide up to reparametrization of the time variable. We point out that under stereo-graphic projection the norm on $T^*S^2$ induced by the round metric becomes $K(q,p)$ but with the following identification: $p$ corresponds to the $S^2$ coordinate and $q$ to the fiber coordinate. We note that the roles of $q$ and $p$ are exchanged: on $T^*S^2$ the $p$ variable is the position variable and $q$ is the momentum (fiber) variable. Indeed, the points in the fiber over the point at infinity correspond precisely to collisions since there the (physical) momentum variables explode.

More generally for energies $E=-k<0$ we set
\beq
K_k(q,p):=|q|\big(E(q,p)+k)+1=\tfrac12\big(|p|^2+2k)|q|\;.
\eeq
This gives again rise to the geodesic flow of the round metric (up to multiplication by a constant) but in disguise. Indeed, if we make the following symplectic change of coordinates: $(q,p)\mapsto(\frac{1}{\sqrt{2k}}q,\sqrt{2k}p)$ we obtain
\beq
K_k(\tfrac{1}{\sqrt{2k}}q,\sqrt{2k}p)=\sqrt{2k}\;K_{\frac12}(q,p)\;.%=\tfrac12\big(|\sqrt{k}p|^2+k)|\tfrac{q}{\sqrt{k}}|=\tfrac{\sqrt{k}}{2} \big(|p|^2+1)|q|\;.
\eeq
In particular, the regularized energy hypersurfaces coincide with the unit cotangent bundle of $S^2$ which is diffeomorphic to $\RP^3$.

\section{The rotating Kepler problem}

The rotating Kepler problem is the Kepler problem in a rotating coordinate system.  More precisely, we regard the rotating Kepler problem as the limit of the planar circular restricted three body problem in which a lighter primary orbits a heavier primary in a circular clockwise direction of constant unit angular speed, and the mass of the smaller primary tends to zero; we then apply a time dependent change of coordinates which results in the heavier primary being fixed at $0\in \mathbb{C}$ and the lighter (in fact massless) primary being fixed at $1\in \mathbb{C}$.

Note that the angular momentum $L(q,p)=q_1p_2-q_2p_1$ generates counter-clockwise rotation around the origin with constant unit angular speed, and this rotation commutes with the Kepler flow; consequently it can be shown that the Hamiltonian of the rotating Kepler problem is given by
\beq\label{eqn:H=E+L}
H=E+L\;.
\eeq

We fix some conventions. The symplectic form is given by $\om=\sum dp_i\wedge dq_i$ and the Hamiltonian vector field of $L$ is given by $\om(X_L,\cdot)=-dL$.

We write
\beq
H(q,p)=\tfrac12|p_1-q_2|^2+\tfrac12|p_2+q_1|^2+U(q)
\eeq
where
\beq
U(q)=-\frac{1}{|q|}-\tfrac12|q|^2
\eeq
denotes the effective potential. The Hamiltonian $H$ is not of mechanical form, that is, kinetic energy plus potential. Instead $H$ belongs to the class of magnetic Hamiltonians and the Lorentz force corresponds to the Coriolis force induced by the rotating coordinate system. The second term in the effective potential is responsible for the centrifugal force.

We denote by $\pi:T^*(\R^2\setminus\{0\})\to\R^2\setminus\{0\}$ the projection to the position space. Then for each $c\in\R$, the associated Hill's region $\mathcal{H}_c\subset\R^2\setminus\{0\}$ is defined to be
\beq
\mathcal{H}_c:=\pi\big(H^{-1}(-c)\big)\;.
\eeq
We note that
\beq
\mathcal{H}_c=\{q\mid U(q)\leq -c\}\;.
\eeq
The value $c$ is \emph{minus} the Jacobi energy. We use this notation in order to keep in line with tradition: the original definition of the Jacobi integral differs from our choice of Hamiltonian $H$ by a minus sign and a factor $2$.

\begin{figure}[htb]
\includegraphics[width=0.3\textwidth,clip]{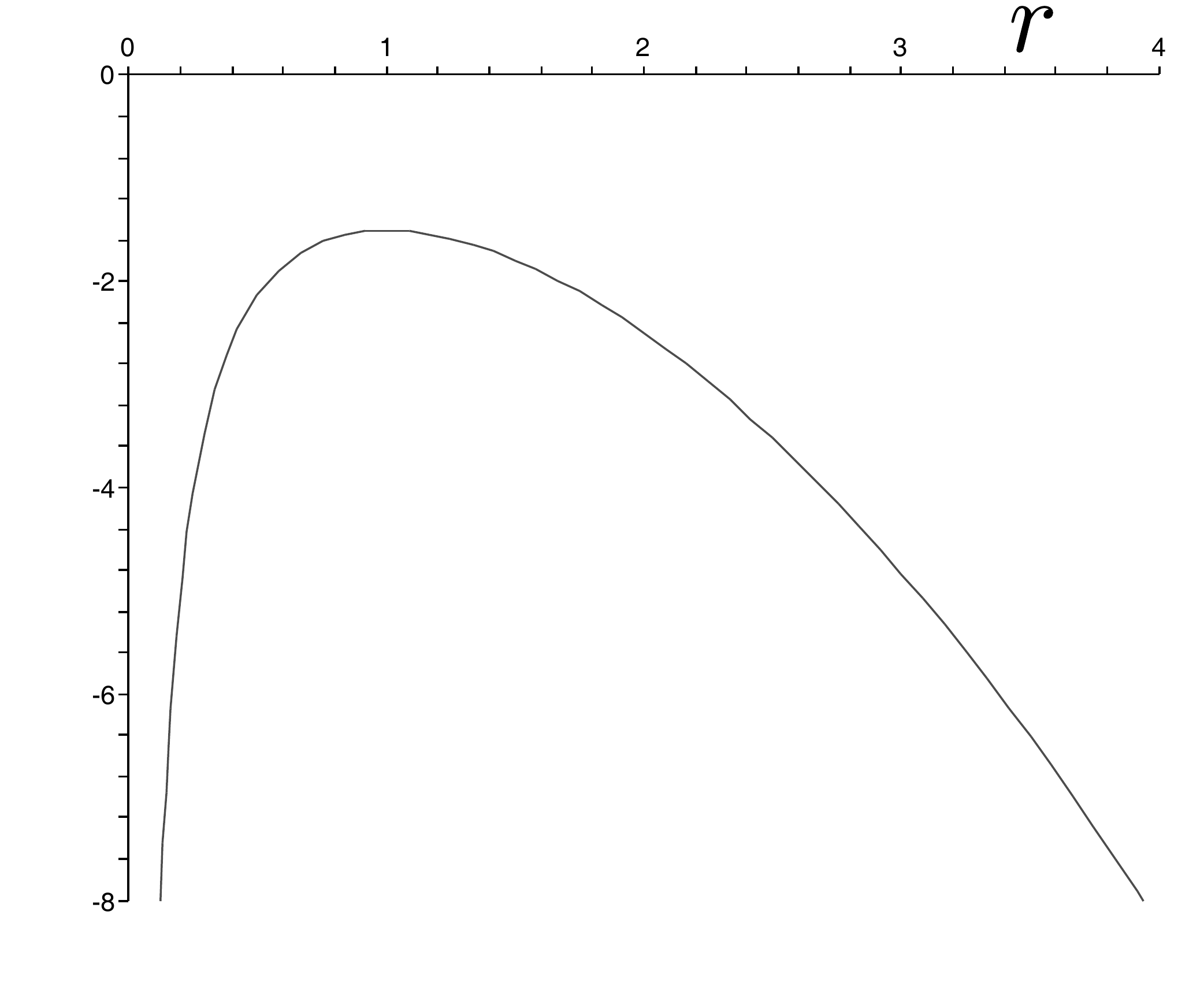}
\caption{Plot of the effective potential $-\frac{1}{r}-\frac{1}{2}r^2$}\label{fig:effective_pot}
\end{figure}
Consider the effective potential $r\mapsto-\frac{1}{r}-\frac{1}{2}r^2$; see Figure~\ref{fig:effective_pot} for its graph.  Since this function attains its unique maximum value of $-\frac32$ at $r=1$, it follows that for all $c>\tfrac32$ Hill's region $\mathcal{H}_c$ is comprised of two connected components: one is bounded and the other is unbounded. If $c\leq\tfrac32$ then Hill's region coincides with $\R^2\setminus\{0\}$. For $c>\tfrac32$ we let
\beq
\mathcal{H}_c^b
\eeq
denote the bounded component of Hill's region $\mathcal{H}_c$. Moreover, we define
\beq
\Sigma_c:=\pi^{-1}\big(\mathcal{H}_c^b\big)\subset H^{-1}(-c)\;.
\eeq
As in the inertial Kepler problem one can apply Moser regularization, see for instance \cite{Albers_Frauenfelder_Koert_Paternain_Liouville_field_for_PCR2BP}. We denote the regularized energy hypersurface by $\Sb_c$. It is again diffeomorphic to $\RP^3$. From now on we only consider the regularized system.

Note that in general, periodic orbits of the inertial Kepler problem will not give rise to periodic orbits of the rotating Kepler problem.  More precisely, if $\gamma:\R\to \mathbb{R}^4$ solves the inertial Kepler problem, and $\Phi_t:\mathbb{R}^4\to\mathbb{R}^4$ denotes the time-dependent change of coordinates from the inertial problem to the rotating problem, then $\alpha(t):=\Phi_t \gamma(t)$ will solve the rotating Kepler problem; however if $\gamma$ is periodic, then in general $\alpha$ will \emph{not} be periodic. Indeed, this is precisely due to the fact that $\Phi_t$ is time-dependent. Moreover, contrary to the inertial Kepler problem, the flow of the rotating Kepler problem is no longer periodic.

Despite all this, there are two cases in which periodic orbits of the inertial problem yield periodic orbits for the rotating problem.  We now specify these cases.

The first is the case of circular orbits.  Recall that an orbit of the rotating Kepler problem is circular if and only if it traces out a circular path in the inertial frame; however, since the coordinate change $\Phi_t$ is a time-dependent rotation in both the $q$ and $p$-plane, we see that $\Phi_t$ takes circular trajectories to circular trajectories, and hence such orbits also trace out circular paths in the rotating frame.  In fact, every circular periodic orbit of the rotating Kepler problem can be constructed from a circular periodic orbit of the inertial Kepler problem, although with different period.

The second case consists of those trajectories which trace out elliptical paths of positive eccentricity in the inertial frame; in this case a $T$-periodic orbit of the inertial problem yields a periodic orbit of the rotating problem if and only if $T$ is a rational multiple of $2\pi$.\footnote{The period in the inertial coordinate system is called \emph{sidereal period}.} Note that this scenario also covers the case that the elliptic orbit in the inertial frame is degenerate\footnote{These are precisely the so-called collision orbits.}, or equivalently has eccentricity $1$.

Observe that the above discussion can be turned around to provide a useful characterization of all periodic orbits of the rotating problem.  More specifically, a $T$-periodic solution of the rotating Kepler problem either traces out a circle in the inertial frame, or else it multiply covers a (possibly degenerate) ellipse of positive eccentricity in the inertial frame.  In the latter case it can be shown that $T$ is then an integer multiple of $2\pi$. In fact, this characterization motivates the following definition.

\begin{definition}
We say a $T$-periodic orbit $\alpha:\mathbb{R}/T\mathbb{Z}\to \mathbb{R}^4$ of the rotating Kepler problem is a $k$-fold covered ellipse\footnote{Here the ellipse is allowed to have all possible eccentricities, including $1$ (the degenerate collision orbits) and $0$ (the circular orbits).} in an $l$-fold covered coordinate system provided the following hold.
\begin{itemize}
\item There exists positive $l\in \mathbb{N}$ such that $T=2\pi l$, and
\item the corresponding trajectory in the inertial coordinate system given by $\gamma(t):=\Phi_t^{-1}\alpha(t)$ is a $k$-fold covered ellipse of the standard Kepler problem.
\end{itemize}
\end{definition}

%%%%%%%%%%%%%%%%%%%%%%%%%%%%%%%%%%%%%%%%%%%%%%%%%%%%%%%%%%%%%%%%%%%%%%%%%%%%%
%%% Removed by jwf; it was unclear whether this added to the discussion.
%To keep track of the orbits involved, it is useful to observe the following:
%since $E$ and $L$ Poisson-commute, all orbits in the rotating Kepler problem (Hamiltonian given by $H=E+L$) can be %thought of as the composition of the Kepler flow (Hamiltonian given by $E$) with the ``angular momentum'' flow %(Hamiltonian given by $L$).
%Orbits of the Kepler flow with negative energy are always ellipses if we allow eccentricity $0$ for circular orbits %and $1$ for collision orbits.
%%%%%%%%%%%%%%%%%%%%%%%%%%%%%%%%%%%%%%%%%%%%%%%%%%%%%%%%%%%%%%%%%%%%%%%%%%%%

It is worth mentioning that in the rotating Kepler system, periodic orbits can no longer be interpreted as geodesics of a Riemannian metric on $S^2$ but instead as geodesics of a Finsler metric; further details can be found in \cite{Cieliebak_Frauenfelder_vKoert_Cartan_geometry_of_the_rotating_Kepler_system}.
%Closed geodesics always arise in circle families where the circle corresponds to a choice of an initial point on the closed geodesic. Since the rotating Kepler Hamiltonian is still invariant under rotations of the coordinate system there is an additional $S^1$-action. Thus we have a $T^2$-action on periodic orbits.

Since the image of each circular periodic orbit is fixed under the $S^1$-action which rotates the coordinate system, it follows that each circular orbit gives rise to an $S^1$-family of periodic orbits; geometrically this $S^1$-family forms a single circle.  By contrast, ellipses of positive eccentricity in an inertial system can form $T^2$ families of periodic orbits.
Introduce the following notation if this happens;
let $T_{k,l}$ denote the torus comprised of $k$-fold covered ellipses in an $l$-fold covered rotating coordinate system.
Using Delaunay coordinates, one can see that these tori are of Morse-Bott type; see \cite{Barrar_Existence_of_periodic_orbits_of_the_second_kind_in_the_restricted_problem_of_three_bodies} for further details.

These tori play a prominent role in the theory of Kirkwood gaps in the asteroid belt of the Sun-Jupiter system. The first few carry names\footnote{Hilda is the name of the eldest daughter of the astronomer Theodor von Oppolzer.} as follows.

%\begin{center}
%\begin{tabular}{c|l}
%$T_{2,1}$ & Hekuba\\
%\hline$T_{3,2}$ & Hilda \\
%\hline$T_{4,3}$ & Thule\\
%\hline$T_{3,1}$ & Hestia\\
%\hline$T_{7,4}$ & Cybele
%\end{tabular}
% \end{center}

 \begin{center}
\begin{tabular}{c|c|c|c|c}
$T_{2,1}$ &$T_{3,2}$ & $T_{4,3}$ & $T_{3,1}$ & $T_{7,4}$\\
\hline Hekuba& Hilda & Thule & Hestia & Cybele
\end{tabular}
 \end{center}

\section{Circular orbits in the rotating Kepler problem}

Recall that our goal is to compute the Conley-Zehnder indices of all periodic orbits of the rotating Kepler problem for all Jacobi energies beneath the first critical value.  In the previous section, we classified all such periodic orbits either as a circular orbit in an $S^1$ family or as an ``elliptic'' orbit in a $T_{k,l}$-torus family.  The purpose of this section is describe how the families of circular orbits change as we vary the Jacobi energy.  In particular, we show that each such circular orbit has $E$-energy which varies smoothly with $c$.

The circular orbits are, by definition, characterized by the vanishing of their eccentricity. We fix an energy hypersurface $\{H=-c\}$ and consider circular orbits. Combining equations \eqref{eqn:eccentricity} and \eqref{eqn:H=E+L} we obtain the cubic equation
\beq\label{eqn:cubic}
0=2E(-c-E)^2+1\;.
\eeq
The solution set is the union of the two graphs in Figure \ref{fig:energy_Jacobi_diagram}.

\begin{figure}[htb]
\includegraphics[width=0.6\textwidth,clip]{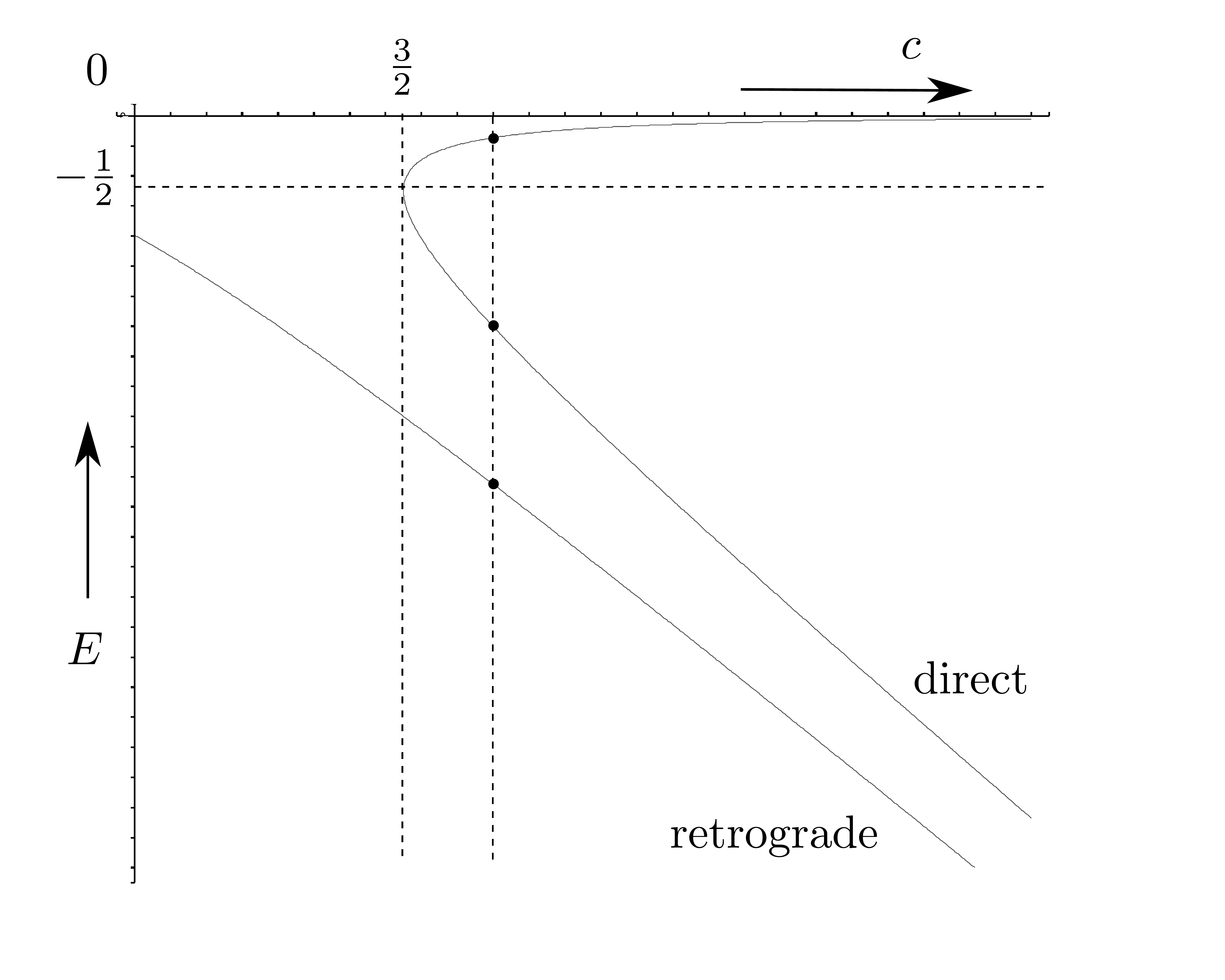}
\caption{The energy-Jacobi diagram}\label{fig:energy_Jacobi_diagram}
\end{figure}

The part of the graph with values $E>-\frac12$ lies in the unbounded component. In particular, for $c>\frac32$ there are two simply covered circular orbits which lie above the bounded component of Hill's region, $\mathcal{H}_c^b$; the one that rotates in the same direction as the coordinate system is rotating is called direct, and the one that rotates in the opposite direction is called retrograde.
In our setup this means that orbits with positive angular momentum $L$ are retrograde and orbits with negative angular momentum are direct.
The third circular orbit lies in the unbounded component of Hill's region.

\section{The life of tori}
In this section we study the behavior of the $T_{k,l}$-torus families of ``elliptic'' orbits as one varies the Jacobi energy.  In particular, we find that unlike the circular orbits, the $E$-energy of the $T_{k,l}$ family does not change as one varies the Jacobi energy $-c$, and furthermore these $T_{k,l}$ families only exist for a finite range of values of $c$. Additionally we shall see that if $T_{k,l}$ is a family of orbits with extremal Jacobi energy then these orbits are in fact circular.  In this way, we shall envision circular orbits as ``giving birth'' to the $T_{k,l}$ families of orbits precisely when the periods of the circular orbits cross multiples of $2\pi$.

Recall that the tori $T_{k,l}$ are obtained from a $k$-fold covered ellipse in a $l$ times rotating coordinate system. We first compute the energy of the ellipse underlying $T_{k,l}$ using Kepler's laws. From the definition of $T_{k,l}$ we obtain the relation
\beq
kT=2\pi l\;.
\eeq
Using $T^2=-\frac{\pi^2}{2E^3}$ we see
\beq
\frac{4\pi^2 l^2}{k^2}=-\frac{\pi^2}{2E^3}
\eeq
and thus
\beq
E_{k,l}=-\frac12\left(\frac{k}{l}\right)^{\frac23}\;.
\eeq
Since we are only interested in energies $E$ below $E_{k,l}<-\frac12$, we shall restrict ourselves from now on to
\beq\label{eqn:k,l_restrictions}
l=1,2,\ldots\quad\text{and}\quad k>l\;.
\eeq

We recall equation \eqref{eqn:cubic} for circular orbits
\beq
0=2E(c+E)^2+1\;.
\eeq
%O
Thus, the possible values of $c$ (minus Jacobi energy) of circular orbits with energy $E_{k,l}$ are
\beq
c^+_{k,l}=-E_{k,l}- \sqrt{\frac{1}{-2E_{k,l}}}\;
\eeq
for the retrograde orbit, and
\beq
c^-_{k,l}=-E_{k,l}+ \sqrt{\frac{1}{-2E_{k,l}}}\;
\eeq
for the direct orbit.
Indeed, as $-c=H=E+L$, we see that the retrograde orbit has positive angular moment $\sqrt{\frac{1}{-2E_{k,l}}}$, whereas the corresponding direct orbit has negative angular momentum $-\sqrt{\frac{1}{-2E_{k,l}}}$.

%O
The synodical periods, i.e.~the periods in the rotating coordinate system, of the circular orbits are
\beq
\label{eq:synodical}
T_r^\pm=\frac{2\pi}{(-2E)^{\frac32}\pm1}
\eeq
where $T_r^+$ corresponds to retrograde circular orbits, and $T_r^-$ corresponds to direct circular orbits.
Note that retrograde orbits have \emph{smaller} period than the direct orbits.

We parametrize the lifetime of the the tori $T_{k,l}$ by decreasing values of $c$ (that is, by increasing the values of the Jacobi energy $H=-c$), see Figure \ref{fig:life_of_tori}. In this way a $T_{k,l}$-torus family of periodic orbits is born out of a multiple cover of a direct circular orbit with Jacobi energy $-c=-c^-_{k,l}$. At this energy, the ellipses in the torus are direct; that is, they have negative angular momentum. This means that the direct orbit is $(k-l)$-fold covered, and this
can be seen as follows. Suppose at Jacobi energy $-c=-c^-_{k,l}$ a torus is born out of a $N^-$-fold cover of a direct circular orbit.
In particular, the period of the $N^-$-fold cover needs to match those of the torus orbits that are born,
\bea
N^- T_r^-&=2\pi l\;.
\eea
We know that the torus $T_{k,l}$ has energy $E_{k,l}=-\frac{1}{2}\left( \frac{k}{l} \right)^{\frac23}$. Furthermore $ T_r^-=\frac{2\pi}{(-2E)^{3/2}-1}$, so we find
$$
N^-=l \Big(\frac{k}{l}-1\Big)=k-l.
$$

As $c$ decreases, the ellipse becomes more and more eccentric until the eccentricity equals $1$, whence the orbit is a collision orbit. If $c$ is decreased further the eccentricity starts to decrease and the ellipse now is retrograde, i.e.~rotates in the opposite direction. Finally, when $c=c^+_{k,l}$ the eccentricity becomes $0$, and the orbits dies in the arms of the $(k+l)$-fold covered retrograde circular orbit.

The lives of Hekuba, Hilda, Thule, Hestia, Cybele, and $T_{7,2}$ are shown in Figure \ref{fig:life_of_tori}.

\begin{figure}[htb]
\includegraphics[width=0.6\textwidth,clip]{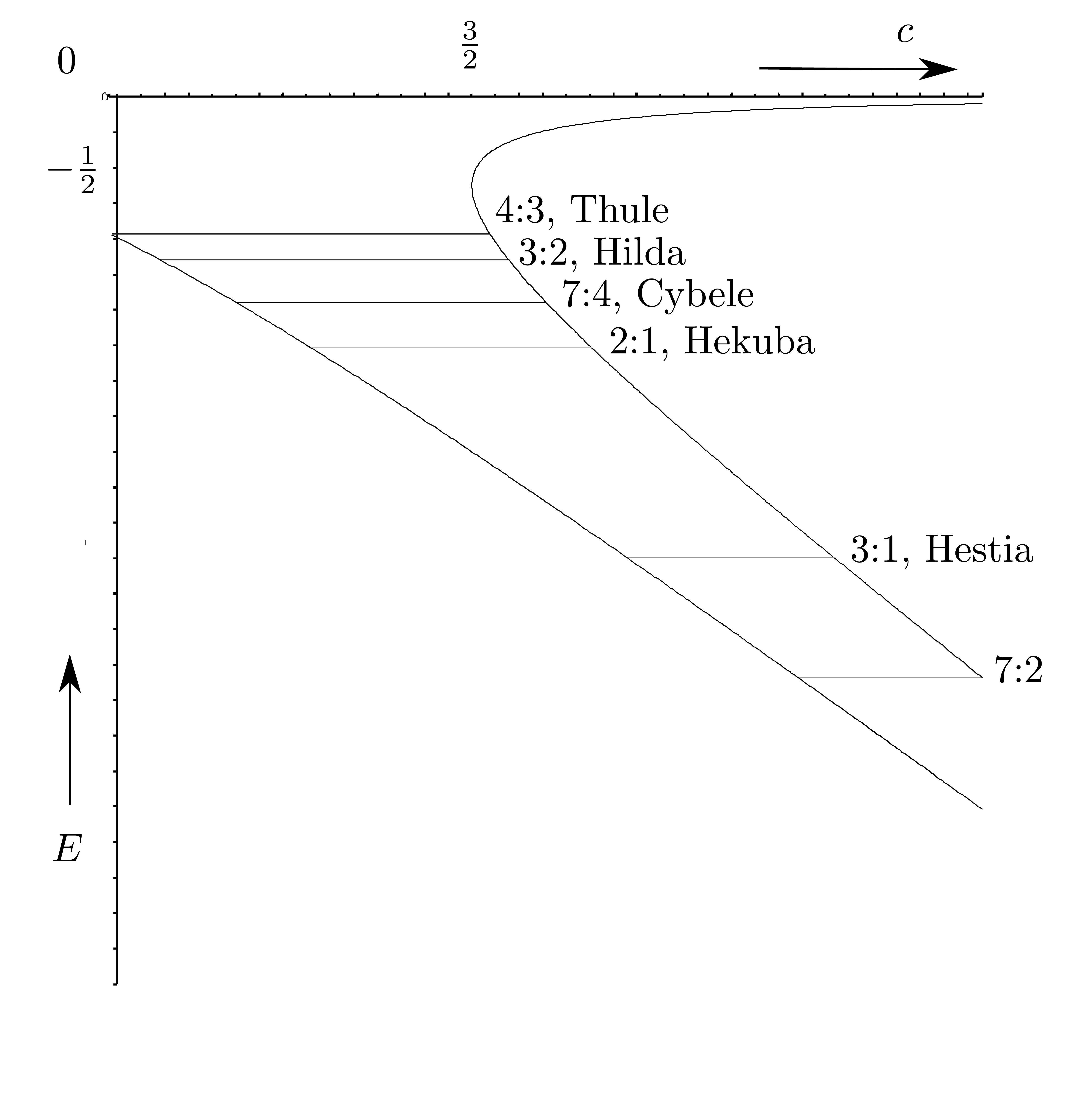}
\caption{The life of tori}\label{fig:life_of_tori}
\end{figure}

\section{Main argument}
In this section we provide the proof of Theorem \ref{thm:main}. However before doing so we first present certain key concepts and then provide a relevant example which illustrates the main proof technique.

The first important idea for the proof is that we shall not consider periodic orbits for a \emph{fixed} value of Jacobi energy, but rather we consider families of orbits that arise from varying $c$ as well.  In this way, any two non-degenerate orbits (of possibly different Jacobi energies) which are connected via a path of non-degenerate orbits must have the same Conley-Zehnder index.  As it turns out, (see Proposition \ref{prop:index_circular_orbit} below) the only orbits which fail to be Morse-Bott non-degenerate are those at bifurcation points; that is, only at those orbits which are both circular and are in a $T_{k,l}$-torus family.

The second key point is that the energy surfaces considered here give rise to Finsler metrics on $S^2$, and that periodic orbits of the rotating Kepler problem are in fact critical points of the energy functional associated to this Finsler metric.  In other words, periodic orbits can be regarded as Finsler-geodesics. Consequently we can assign a Morse index to each periodic orbit, and more importantly, the Conley-Zehnder index considered here and this Morse index agree; see for instance \cite{Duistermaat_On_the_Morse_index_in_variational_calculus, Weber_Peridodic_closed_geodesics_are_periodic_orbits:index_and_transversality, Abbondandolo_On_the_Morse_index_of_Lagrangian_systems}.

To see the third key point, we first recall that the Morse index of a degenerate orbit of Morse-Bott type is defined as the number of negative eigenvalues of the energy functional at that orbit. A consequence of this definition is that the Morse index cannot decrease after a small perturbation; this is the third key point, and it is a fact which we will exploit in our proof.

\subsection{The births of Hekuba, Hilda, and subsequent siblings}
With the key ingredients established, we now move on to our example. We begin by considering the direct circular orbits which wind around the origin (in the rotating coordinate system) precisely once.  As mentioned at the beginning of this section, these orbits are Morse-Bott non-degenerate whenever they are not also a $T_{k,l}$-type orbit.  Observe that in our example the winding condition guarantees that the circular orbits we are considering cannot be be of $T_{k,l}$-type unless $k=l+1$. We conclude that these circular orbits are only degenerate when their Kepler energy is precisely
\begin{equation}\label{eq:windingOneEnergies}
E_{k,k-1}=-\frac12 \left(\frac{k}{k-1}\right)^{\frac{2}{3}}.
\end{equation}
Observe that these Kepler energies accumulate at $-\frac12$, and are minimal at $E_{2,1}$.  Consequently, the circular orbits we are considering are non-degenerate whenever their Kepler energy is less than $E_{2,1}=-\frac12(2)^{\frac{2}{3}}$, or equivalently whenever $c>c_{2,1}^-\approx 1.59$. Furthermore, a direct computation (specifically Proposition \ref{prop:index_circular_orbit} below) shows that for very large values of $c$ (i.e. very negative Kepler energy) such orbits have Conley-Zehnder index equal to $3$.

We now consider what happens when one follows these circular orbits from very large values of $c$ to smaller values of $c$ (or equivalently from very negative Kepler energies to less negative Kepler energies).  Indeed, in this case the Conley-Zehnder index remains $3$ until $c=c_{2,1}^-$ (or equivalently at $E=E_{2,1}$) at which point two things happen.  First, at this energy level these simple direct circular orbits give birth to the Hekuba (i.e. $T_{2,1}$) orbits, and second, as $c$ decreases to just below $c_{2,1}^-$ the Morse index (and hence the Conley-Zehnder index) increases from $3$ to $5$.  This latter point is proved in Proposition \ref{prop:index_circular_orbit}.

To compute the Conley-Zehnder index of the Hekuba orbits, we first make use of the fact that all non-circular orbits of the rotating Kepler problem are Morse-Bott non-degenerate\footnote{Recall the Delaunay coordinates and \cite{Barrar_Existence_of_periodic_orbits_of_the_second_kind_in_the_restricted_problem_of_three_bodies}.} so it is sufficient to determine the Conley-Zehnder index of just one Hekuba orbit.  To that end, we make use of the fact that our periodic orbits are critical points of the energy functional associated to a Finsler metric, and that the Conley-Zehnder index will agree with the associated Morse index.  It then follows by local invariance of Morse homology that the Morse index (and hence the Conley-Zehnder index) of every Hekuba orbit is $3$.  
  
If we continue to follow the branch of circular orbits (more specifically, the direct circular orbits which wind around the origin precisely once) through decreasing values of $c$, then we find that the Conley-Zehnder index remains $5$ until we reach the energy level $c=c_{3,2}^-$, at which point we have another bifurcation. At the energy level $c_{3,2}^-$, the direct circular orbits give birth to the Hilda orbits (i.e. the $T_{3,2}$ orbits), and for slightly smaller values of $c$ the Conley-Zehnder index jumps from $5$ to $7$; again by local invariance of Morse-homology and Morse-Bott non-degeneracy of the non-circular orbits we find that the Conley-Zehnder index of every Hilda orbit is $5$.

One can now continue this process, namely decreasing the value of $c$ as close as we like to $\frac{3}{2}$, and each time the value of $c$ crosses one of the values $\{c_{k,k-1}^-\}_{k>1}$ the Conley-Zehnder index increases by $2$ and an additional family of $T_{k,k-1}$-siblings is born.  The Morse-Bott non-degeneracy of non-circular orbits and invariance of local Morse-homology determines the Conley-Zehnder index of all such $T_{k,k-1}$-type orbits.

The above argument essentially computes the Conley-Zehnder indices for all $T_{k,k-1}$-type orbits.  Since Proposition \ref{prop:index_circular_orbit} below computes the Conley-Zehnder indices for all circular orbits, it will be useful to compute the indices for the more general $T_{k,l}$-type orbits.  To that end, we first observe that each $T_{k,l}$-type orbit is born out of a circular direct orbit which winds around the origin precisely $(k-l)$ times, or equivalently a $(k-l)$-covered circular orbit of the rotating Kepler problem.  Again making use of Proposition \ref{prop:index_circular_orbit}, it follows that for very large $c$ the $(k-l)$-fold covered direct circular orbit has Morse index $2(k-l)+1$. At each birth the index increases by $2$. In particular, before giving birth to $T_{k,l}$ the $(k-l)$-fold covered direct circular orbit has Morse index
\beq
2(k-l)+1+2(l-1)=2k-1\;.
\eeq
After the birth of $T_{k,l}$ the torus $T_{k,l}$ acquires this Conley-Zehnder index by invariance of local Morse homology. This proves the following claim.
\begin{lemma}\label{lem:CZofTkl}
Assume $k,l\in \mathbb{N}$ with $k>l\geq 1$.  Then the Conley-Zehnder index of each $T_{k,l}$-type orbit is equal to $2k-1$.
\end{lemma}

\subsection{Proof of Theorem \ref{thm:main}}
We need to show that the Conley-Zehnder indices of contractible periodic orbits are greater or equal to $3$. 

\noindent{\bf Claim:}\emph{The circular orbits are contractible if and only if they are evenly-covered.}
Indeed, since every even cover of a loop in $\R P^3$ lifts to a loop $S^3$, we see that all evenly-covered orbits are contractible.
To see that odd covers of simple circular orbits are not contractible, observe that it suffices to show that a simply covered circular is not contractible.
This is the case because one can construct a homotopy of circular orbits by varying the energy level $c$. For $c\to \infty$ a circular orbit in the rotating Kepler problem becomes close to a simple orbit of the geodesic flow on $S^2$: such orbits are not contractible.

Because these circular orbits are evenly covered, it follows from Proposition \ref{prop:index_circular_orbit} that their Conley-Zehnder indices are at least $3$. Furthermore it follows from Lemma \ref{lem:CZofTkl} that the Conley-Zehnder indices of each $T_{k,l}$-type orbit (contractible or not) is at least $3$.  This completes the proof of Theorem \ref{thm:main}.  We finish this section with an informative corollary.

\begin{corollary}
For each $c>\frac{3}{2}$ the doubly covered retrograde orbit is the unique contractible periodic orbit of Conley-Zehnder index $3$.
\end{corollary}
\begin{proof}
As mentioned previously, the only contractible orbits are those which are evenly covered.  Observe that Proposition \ref{prop:index_circular_orbit} guarantees that the doubly covered circular retrograde orbits are the unique contractible \emph{circular} orbits with Conley-Zehnder index $3$.  Recall that the covering number of $T_{k,l}$-type orbit is given by $(k-l)$, and the condition that $k>l\geq 1$ with $k,l\in \mathbb{N}$ guarantees that if $(k-l)$ is even, then $k\geq 3$.  It then follows from Lemma \ref{lem:CZofTkl} that the Conley-Zehnder index of any contractible $T_{k,l}$-type orbit is at least $5$.
\end{proof}

\section{Proof of Theorem \ref{thm:main2}}

We recall that the Levi-Civita coordinates are given by
$q=2v^2$ and $p=\frac{u}{\bar{v}}$ in \cite{Levi_Civita}.  These coordinates define  a 2:1-map, which is symplectic up to a factor $4$.  Indeed, $\Re(dq\wedge d\bar{p})=4\Re(dv\wedge d\bar{u})$. Transforming and regularizing the Hamiltonian function $H(q,p)$ from equation \eqref{eqn:H=E+L} at energy $-c$ leads to
\beq
K_{c}(u,v):=|v|^2\big(H(u,v)+c\big)=\frac12|u|^2+c|v|^2+2|v|^2\langle u,iv\rangle -\tfrac12\;.
\eeq
A component of the energy hypersurface $H^{-1}(c)$  lifts to a compact component $\Sigma_{c}$ of the energy hypersurface $K_{c}^{-1}(0)$ which is diffeomorphic to $S^3\subset\C^2$.

Using complex notation the gradient and Hessian of $K_c$ are given by the following.
\bea
DK_c(u,v)(\uh,\vh)=&\langle u,\uh\rangle +2c\langle v,\vh\rangle +4\langle v,\vh\rangle \langle u,iv\rangle \\
&+2|v|^2\langle \uh,iv\rangle +2|v|^2\langle u,i\vh\rangle 
\eea

\bea
D^2K_c(u,v)((\uh,\vh),(\uh,\vh))=&|\uh|^2+2c|\vh|^2+4\langle u,iv\rangle |\vh|^2\\
&+8\langle v,\vh\rangle \langle u,i\vh\rangle +8\langle v,\vh\rangle \langle \uh,iv\rangle \\
&+4|v|^2\langle \uh,i\vh\rangle 
\eea

We fix throughout the remaining part the value of $c$ such that it corresponds to the critical value of the Jacobi energy: $c=\tfrac32$ and set $K:=K_{3/2}$. For the point $(u,v)=(-ia,\tfrac12)$ to lie on the energy hypersurface
\beq
\{K(u,v)=\tfrac12|u|^2+\tfrac32|v|^2+2|v|^2\langle u,iv\rangle -\tfrac12=0\}
\eeq
we derive for $a\in\R_{>0}$
\bea
0=K(-ia,\tfrac12)&=\tfrac12a^2+\tfrac32\tfrac14-2\tfrac14a\tfrac12-\tfrac12\\
%&=\tfrac12a^2+\tfrac38-\tfrac14a-\tfrac12\\
&=\tfrac12a^2-\tfrac14a-\tfrac18\\
&=\tfrac12(a^2-\tfrac12a-\tfrac14)\;.
\eea
We choose the zero given by
\beq
a=\tfrac14+\sqrt{\tfrac{1}{16}+\tfrac14}=\tfrac{1+\sqrt{5}}{4}\approx0.80902
\eeq
and note that
\beq
4a-1=\sqrt{5}\;.
\eeq
Next we fix a vector $(\uh,\vh)\in\C^2$ with $\vh=1$ and $\uh\in i\R$ such that
\bea
0=DK(-ia,\tfrac12)(\uh,\vh)&=\langle u,\uh\rangle +3\langle v,\vh\rangle +4\langle v,\vh\rangle \langle u,iv\rangle +2|v|^2\langle \uh,iv\rangle +2|v|^2\langle u,i\vh\rangle \\
&=-a\uh_2+\tfrac32+4\tfrac12(-\tfrac12a)+2\tfrac14\tfrac12\uh_2+2\tfrac14(-a)\\
&=(\tfrac14-a)\uh_2+\tfrac32(1-a)
\eea
and conclude
\beq
\uh_2=\frac{\tfrac32(1-a)}{a-\tfrac14}=\frac{6(1-a)}{4a-1}=\frac{9\sqrt{5}}{10}-\frac32\approx0.51246\;.
\eeq
We observe
\bea
\uh_2^2+3\uh_2&=\left(\frac{9\sqrt{5}}{10}-\frac32\right)\left(\frac{9\sqrt{5}}{10}-\frac32+3\right)=\left(\frac{9\sqrt{5}}{10}-\frac32\right)\left(\frac{9\sqrt{5}}{10}+\frac32\right)\\
&=\frac{81*5}{100}-\frac{9}{4}=\frac{405}{100}-\frac{225}{100}=\frac{180}{100}\\
&=\frac95\;.
\eea
Now we compute
\bea
D^2K(u,v)((\uh,\vh),(\uh,\vh))=&|\uh|^2+3|\vh|^2+4\langle u,iv\rangle |\vh|^2\\
&+8\langle v,\vh\rangle \langle u,i\vh\rangle +8\langle v,\vh\rangle \langle \uh,iv\rangle \\
&+4|v|^2\langle \uh,i\vh\rangle \\
&=\uh_2^2+3-4\tfrac12a\\
&+8\tfrac12(-a)+8\tfrac12(\uh_2\tfrac12)\\
&+4\tfrac14\uh_2\\
%&=\uh_2^2+\uh_2(2+1)+3-2a-4a\\
&=\uh_2^2+3\uh_2+3-6a\\
&=\tfrac95+3-6\tfrac{1+\sqrt{5}}{4}\\
%&=\tfrac{3}{2}\left(\tfrac65+2-(1+\sqrt{5})\right)\\
&=\tfrac{3}{2}\left(\tfrac{11}{5}-\sqrt{5}\right)\;.
\eea
Since
\beq
\left(\tfrac{11}{5}-\sqrt{5}\right)\left(\tfrac{11}{5}+\sqrt{5}\right)=\tfrac{121}{25}-\tfrac{125}{25}<0
\eeq
we conclude that for
\beq
(u,v)=(\tfrac{1+\sqrt{5}}{4i},\tfrac12)\quad\text{and}\quad(\uh,\vh)=\left(i\left(\frac{9\sqrt{5}}{10}-\frac32\right),1\right)
\eeq
we have
\beq
D^2K(u,v)((\uh,\vh),(\uh,\vh))<0\;.
\eeq
In particular, since the Hessian on a tangential direction of $\{K=0\}$ is negative the energy hypersurface $\{K=0\}$ is not convex. By continuity the same remains true for values of $c$ slightly less that $\tfrac32$. This proves Theorem \ref{thm:main2}.

\appendix

\section{The Maslov and Conley-Zehnder index}

\subsection{Definition of a Maslov index using a crossing form}
Here we shall work with the Robbin-Salamon definition of the Maslov index, see \cite{Robbin_Salamon_Maslov_index_for_paths}.

Let $\omega_0$ denote the standard symplectic form on $\R^{2n}$ given by
$$
\omega_0=dx \w dy.
$$
\begin{definition}
Let $\psi:[0,T]\to Sp(2n)$ be a path of symplectic matrices.
We call a point $t\in [0,T]$ a {\bf crossing} if $\det (\psi(t)-\id)=0$.
For a crossing $t$ we define the crossing form as the following quadratic form.
Let $V_t=\ker (\psi(t)-\id)$ and define for $v\in V_t$
$$
Q(v,v):=\omega_0(v,\dot \psi(t) v).
$$
The quadratic form $Q$ is called the {\bf crossing form}.
\end{definition}

Let us now define the Maslov index for symplectic paths in the following steps.
Take a path of symplectic matrices $\psi:[0,T]\to Sp(2n)$ and suppose that all crossings are isolated.
Suppose furthermore that all crossings are non-degenerate, i.e.~the crossing form $Q_t$ at the crossing $t$ is non-degenerate as a quadratic form.
Then we define the Maslov index of $\psi$ as
$$
\mu(\psi)=\frac{1}{2} \sgn Q_0+\sum_{t\in (0,T) \text{ crossing} } \sgn Q_t +\frac{1}{2} \sgn Q_T
$$
Here $\sgn$ denotes the signature of a quadratic form.
For $*=0$ or $T$, $\sgn Q_* = 0$ if $*$ is not a crossing.

According to Robbin and Salamon, $\mu(\psi)$ is invariant under homotopies of the path $\psi$ with fixed endpoints.
For a general path of symplectic matrices $\psi:[0,T]\to Sp(2n)$, we choose a perturbation $\tilde \psi$ of $\psi$ while fixing the endpoints, and we define
$$
\mu(\psi):=\mu(\tilde \psi).$$
This is well defined according to Robbin and Salamon, \cite{Robbin_Salamon_Maslov_index_for_paths}.

To define the Conley-Zehnder index of a Reeb orbit $\gamma$, we choose a spanning disk $D_\gamma$ for $\gamma$ and trivialize the contact structure $\xi$ over $D_\gamma$. The linearized flow along $\gamma$ with respect to that trivialization then gives rise to a path of symplectic matrices, $\psi(t):=TFl^R_t(x)|_{\xi}$.
Then {\bf Conley-Zehnder index} of $\gamma$ is given by
$$
\mu_{CZ}(\gamma):=\mu(\psi).
$$

\begin{remark}
Note that this index differs from the Conley-Zehnder index defined in \cite{HWZ_the_dynamics_on_three_dimensional_strictly_conve_eneergy_surfaces}.
For non-degenerate orbits they coincide, though. More precisely, in the definition of Robbin-Salamon the Conley-Zehnder index is shifted by adding half of the nullity of the periodic orbit. The latter is by definition the dimension of the kernel of the Hessian minus $1$. Subtracting $1$ removes the always present degeneracy due to the autonomous character of the Hamiltonian system.
\end{remark}

\section{Trivialization of star-shaped contact forms on $T^*S^2$}
Consider $T^*S^2\subset \R^3\times \R^3$ with coordinates $(\xi,\eta)$.
Use $\xi$ to denote the base point in $S^2$, and let $\eta$ denote the fiber coordinate. Hence
$$
\xi^2=1,\quad
\xi \cdot \eta=0.
$$
In these coordinates, the canonical $1$-form is given by $\lambda=\eta d\xi$.
Let $K:T^*S^2\to \R$ be a fiberwise star-shaped Hamiltonian, i.e.~$\eta \partial_\eta K>0$.
We claim the contact structure associated with the kernel of $\lambda$ on a regular level set of $K$ admits a global trivialization.
Indeed,
$$
X_1=(\xi\times \eta-\frac{(\xi\times \eta)\cdot \partial_\eta K}{\eta\cdot \partial_\eta K } \eta) \cdot \partial_\eta
,\quad
X_2=-\frac{(\xi\times \eta) \cdot \partial_\eta K }{\eta\cdot \partial_\eta K} \eta\cdot \partial_\eta+ (\xi\times \eta)\cdot \partial_\xi.
$$
lie in the kernel of $\lambda$ and $dK$, and since
$$
d\lambda(X_1,X_2)=|\xi \times \eta|^2\neq 0,
$$
we see that these vectors are linearly independent, so they form a symplectic basis after normalization.
We can define a complex structure $J$ by
$$
JX_1=X_2.
$$

\subsection{Trivialization after stereo-graphic projection}
\label{sec:trivialization_after_stereo_proj}
We denote the stereo-graphic projection by $\Pi$, and its the tangent map by $T\Pi$.
Together with the inverse $\Pi^{-1}$, we find
\begin{align*}
&T\Pi \eta \partial_\eta=q\partial_q, \quad T\Pi (\xi\times \eta) \partial_\eta=J_0q,\\
&T\Pi (\xi\times \eta) \partial_\xi=\frac{1}{4}(p^2+1)^2(J_0q)\partial_p+\frac{1}{2}(p^2+1)|q|(J_0p)\partial_q.
\end{align*}
Here $q=(q_1,q_2)$, $J_0q=(q_2,-q_1)$.
From this it follows that $T\Pi X_1$ is a vector that has only components in the $\partial_q$ direction.

\subsection{Kepler Hamiltonian in polar coordinates and linearized flow}
The Hamiltonian for the rotating Kepler problem with angular momentum $a$ is given by
$$
H_a(q,p)=\frac{1}{2}|p|^2+a(q_1p_2-q_2p_1)-\frac{1}{|q|}=E+aL.
$$
We use polar coordinates $q_1=x \cos y$, and $q_2=x \sin y$, which induces a coordinate change on the cotangent bundle.
The latter can be computed using the corresponding canonical $1$-forms $p_1 dq_1 +p_2 dq_2=rdx+tdy$.
We find $p_1=r\cos y -\frac{t}{x}\sin y$  and $p_2=r\sin y +\frac{t}{x}\cos y$.

The angular momentum $L$ is now given by $L=t$ and the transformed Hamiltonian is
$$
H_a(x,y,r,t)=\frac{1}{2}\left( r^2+\frac{t^2}{x^2} \right)
-\frac{1}{x}+at.
$$
The associated Hamilton vector field has the form
$$
X_{H_a}=\frac{t^2-x}{x^3}\partial_r+r\partial_x+\left( \frac{t}{x^2}+a \right) \partial_y.
$$

\subsubsection{Circular orbits}
\label{sec:circular_orbits}
We shall now look for circular orbits.
Circular orbits have constant $x$, hence we need to impose $r=0$.
In particular, $r$ is constant, so it follows that $t^2=x$.
Hence we find the solutions
$$
\left(
\begin{array}{c}
r\\
t\\
x\\
y
\end{array}
\right)
(s)=
\left(
\begin{array}{c}
0\\
\pm \sqrt{x_0}\\
x_0\\
\left(\frac{\pm 1}{x_0^{3/2}}+a\right)s
\end{array}
\right)
.
$$
\begin{remark}
\label{rem:periods}
For $a=1$ we see that the period of a circular orbit is either
$$
\pm \frac{2\pi}{\frac{\pm 1}{x_0^{3/2}}+1},
$$
which can also be expressed in terms of the energy, see Equation~\eqref{eq:synodical}.
The retrograde orbit has positive angular momentum $t=\sqrt x_0$, and the direct orbit has angular momentum $t=-\sqrt x_0$.
Note that the latter interpretation depends on the sign of $a$, which we have taken to be positive.
\end{remark}

\subsection{Linearized equations}
Let us now linearize the equations near a circular orbit.
We shall do this by expanding $(r+\Delta r,t+\Delta t,x+\Delta x,y+\Delta y)$ near $(0,t_0=\pm \sqrt x_0,x_0,y)$.
This leads to the linearized equations (keep in mind that $y$-term in the total flow has a $0$-th order contribution).
\begin{equation}
\label{eq:linearized_flow}
\left(
\begin{array}{c}
\dot{\Delta r} \\
\dot{\Delta t} \\
\dot{\Delta x} \\
\dot{\Delta y}
\end{array}
\right)
=
\left(
\begin{array}{cccc}
0 & \frac{2t_0}{x_0^{3}} & -\frac{1}{x_0^3} & 0 \\
0 & 0 & 0 & 0 \\
1 & 0 & 0 & 0 \\
0 & \frac{1}{x_0^2} & -\frac{2t_0}{x_0^{3}} & 0
\end{array}
\right)
\left(
\begin{array}{c}
\Delta r \\
\Delta t \\
\Delta x \\
\Delta y
\end{array}
\right)
\end{equation}
Note that these linearized equations are autonomous.

\subsection{Trivialization of the contact structure in the unregularized problem}
In order to compute the Maslov index we choose a convenient trivialization of the contact structure.
The canonical $1$-form in the unregularized problem is given by $\lambda=-qdp$.
In terms of polar coordinates for $q$, this becomes
$$
\lambda=-xdr+tdy.
$$
The Hamiltonian for the rotating Kepler problem is given by
$$
H=\frac{1}{2}(r^2+\frac{t^2}{x^2})+at-\frac{1}{x}.
$$
We find a trivialization of $\ker \lambda|_{H^{-1}(-c)}$ by looking at the $\ker \lambda \cap \ker dH$.
We shall choose
$$
\tilde X_1=  \frac{t}{x}\partial_r -\frac{rx}{tx+1}\partial_t+\frac{rx^3}{tx+1}\partial_x+\partial_y,
\quad
\tilde X_2=-\frac{(x-t^2)}{x(tx+1)}\partial_t+\frac{(x^2+t)}{tx+1}\partial_x.
$$
In Cartesian coordinates the vector $\tilde X_1$ has no components in the $\partial_p$ direction, so it is a multiple of $T \Pi X_1$, see the observation in Section~\ref{sec:trivialization_after_stereo_proj}.
Since $d\lambda(\tilde X_1,\tilde X_2)\neq 0$ for $|q|\neq 0$, we see that the pair $(\tilde X_1,\tilde X_2)$ trivializes the contact structure of the rotating Kepler problem.
Furthermore, away from $|q|=0$, this trivialization has the same homotopy class as the global trivialization $(X_1,X_2)$, so we can compute everything in terms of $(\tilde X_1,\tilde X_2)$.

\subsection{Computation of the Maslov index}

\begin{proposition}\label{prop:index_circular_orbit}
Let $\gamma_{+}$ be the simple retrograde circular orbit of the rotating Kepler problem.
Then the (unregularized) period of $\gamma_+$ is equal to
$$
S_+=\frac{2\pi}{(-2E)^{3/2}+1}.
$$
Suppose that $NS_+\notin \Z \frac{2\pi}{(-2E)^{3/2}}$.
Then $N$-th iterate of $\gamma_+$ is non-degenerate and its Conley-Zehnder index is equal to
$$
\mu( \gamma_{+,N})=
1+2 \max \Big\{ k\in \Z \mid k\frac{2\pi}{(-2E)^{3/2}}<NS_+ \Big\}.
$$

Similarly, let $\gamma_-$ be the simple direct circular orbit of the rotating Kepler problem.
Let $S_-$ denote the period of $\gamma_-$,
$$
S_-=\frac{2\pi}{(-2E)^{3/2}-1}.
$$
Suppose that $NS_-\notin \Z \frac{2\pi}{(-2E)^{3/2}}$.
Then the $N$-th iterate of $\gamma_-$ is non-degenerate and its Conley-Zehnder index is equal to
$$
\mu( \gamma_{-,N})=
1+2 \max \Big\{ k\in \Z \mid k\frac{2\pi}{(-2E)^{3/2}}<NS_- \Big\}
.
$$
\end{proposition}

\begin{proof}
The periods of the retrograde and direct orbit have already been computed in section~\ref{rem:periods}.
Throughout the proof, we shall use that $t_0^2=x_0$ and $r_0=0$ at circular orbits.
Hence we see that, at the circular orbit, the tangent space to a level set of $H_a$ is trivialized by
\begin{align*}
&\tilde X_1=\frac{t_0}{x_0}\partial_r +\partial_y=\frac{1}{t_0}\partial_r +\partial_y,\quad \tilde X_2=\frac{x_0^2+t_0}{x_0t_0+1}\partial_x=t_0 \partial_x,\\
& X_H =(\frac{t_0}{x_0^2}+1) \partial_y=(\frac{1}{t_0^3}+1) \partial_y.
\end{align*}
We compute the linearized flow from Equation~\eqref{eq:linearized_flow} with respect to this trivialization.
We obtain
$$
\dot {\tilde \psi}=\tilde L \tilde \psi,
$$
where
$$
\tilde L=
\left(
\begin{array}{ccc}
0 & -\frac{1}{t_0^4} & 0 \\
\frac{1}{t_0^2} & 0 & 0\\
0 & -\frac{t_0^4}{\frac{1}{t_0^3}+1} & 0
\end{array}
\right)
.
$$
Since we have linearized the Hamiltonian rather than the Reeb vector field, we need to project to the contact structure spanned by $\tilde X_1,\tilde X_2$.
For this just take the top-left $2\times 2$-block of $\tilde L$.
This yields the map
$$
\tilde \psi|_{\xi}(s)=
\left(
\begin{array}{cc}
\cos (\frac{s}{t_0^{3}}) &  -\frac{1}{t_0}\sin (\frac{s}{t_0^{3}})\\
t_0 \sin (\frac{s}{t_0^{3}}) & \cos (\frac{s}{t_0^{3}})
\end{array}
\right)
.
$$
This path of symplectic matrices has crossings at $s\in t_0^{3}2\pi \Z$.
Note that $E=-\frac{1}{2x_0}=-\frac{1}{2t_0^2}$.
Since the crossing form has signature $2$, we obtain the claim.
\end{proof}

For a geometric interpretation of the above computation we include the following remark.
\begin{remark}
On the round $2$-sphere an $N$-fold cover of a primitive closed geodesic has Morse index $2N-1$. Moreover, they form critical manifolds, the unit tangent bundle, which are diffeomorphic to $\RP^3$. After switching on the rotation $a$ the $\RP^3$ breaks up into two circles, corresponding to the direct and retrograde circular orbit. Hence, for very large $c$ the $N$-fold cover of the retrograde circular orbit has Morse index $2N-1$ and the $N$-fold cover of the direct circular orbit has Morse index $2N+1$.
\end{remark}

\bibliographystyle{amsalpha}
\bibliography{bibtex_paper_list}

%\begin{thebibliography}{999}
%
%%%%%% \bibitem{AFKP}
%%%%% P.\,Albers, U.\,Frauenfelder, O.\,van Koert, G.\,Paternain, {\em
%%%%%%%%%%% The contact geometry of the restricted 3-body problem},
%%% arXiv:1012.2140.
%\bibitem{CFK} K.\,Cieliebak, U.\,Frauenfelder, O.\,van Koert,
%{\em The Cartan geometry of the rotating Kepler problem}, preprint.
%\bibitem{D} J.\,Duistermaat, {\em On the Morse index in variational
%calculus}, Advances in Math.\,\textbf{21}:2 (1976), 173--195.
%\bibitem{M} J.~Moser, {\em Regularization of Kepler's problem and the
%averaging method on a manifold}, Comm.\,Pure Appl.\,Math.\textbf{23}
%(1970) 609-636.
%\bibitem{W} J.\,Weber, {\em Perturbed closed geodesics are periodic
%orbits:\,index and transversality}, Math.\,Z.\,\textbf{241}:1
%(2002), 45--82.
%\end{thebibliography}

\end{document}